\documentclass[]{interact}

\usepackage{graphicx}
\theoremstyle{plain}
\newtheorem{theorem}{Theorem}[section]
\newtheorem{lemma}[theorem]{Lemma}

\newtheorem{proposition}[theorem]{Proposition}

\theoremstyle{definition}

\theoremstyle{remark}

\numberwithin{equation}{section}

\begin{document}


\title{Jacobi-Angelesco multiple orthogonal polynomials on an $r$-star\thanks{Supported by FWO research project G.0864.16N and EOS project PRIMA 30889451.}}
\author{\name{Marjolein Leurs and Walter Van Assche\thanks{CONTACT: Walter Van Assche. Email: walter.vanassche@kuleuven.be}}
\affil{Department of Mathematics, KU Leuven, Celestijnenlaan 200B box 2400,\\ BE-3001 Leuven, Belgium.}}

\maketitle

\begin{abstract}
We investigate type I multiple orthogonal polynomials on $r$ intervals which have a common point at the
origin and endpoints at the $r$ roots of unity $\omega^j$, $j=0,1,\ldots,r-1$,
with $\omega = \exp(2\pi i/r)$. We use the weight function $|x|^\beta (1-x^r)^\alpha$, with $\alpha,\beta  >-1$ for the multiple orthogonality relations.
We give explicit formulas for the type I multiple orthogonal polynomials, the coefficients in the recurrence relation, the differential equation, and we obtain the asymptotic distribution of the zeros.
\end{abstract}


\begin{amscode}
33C45; 42C05.
\end{amscode}

\section{Introduction}
Various families of multiple orthogonal polynomials have been worked out during the past few decennia, even though
the notion of multiple orthogonality goes back at least to Hermite in the framework of Hermite-Pad\'e approximation.
There are two types of multiple orthogonal polynomials. Let $\vec{n}=(n_1,n_2,\ldots,n_r)$ be a multi-index of size
$|\vec{n}| = n_1+n_2+\cdots+n_r$ and let $\mu_1,\ldots,\mu_r$ be positive measures for which all the moments exist.
Type I multiple orthogonal polynomials for $(\mu_1,\ldots,\mu_r)$ are given by a vector $(A_{\vec{n},1}, \ldots, A_{\vec{n},r})$
of $r$ polynomials, with $\deg A_{\vec{n},j} = n_j-1$, such that the following orthogonality conditions hold:
\[   \sum_{j=1}^r  \int x^k A_{\vec{n},j}(x) \, d\mu_j(x) = 0, \qquad 0 \leq k \leq |\vec{n}|-2, \]
with normalization
\[    \sum_{j=1}^r   \int x^{|\vec{n}|-1} A_{\vec{n},j}(x) \, d\mu_j(x) = 1. \]
The type II multiple orthogonal polynomial for the multi-index $(n_1,\ldots,n_r)$ is the \textit{monic} polynomial $P_{\vec{n}}$
of degree $|\vec{n}|$ for which the following orthogonality conditions hold:
\[    \int x^k P_{\vec{n}}(x)\, d\mu_j(x) = 0, \qquad   0 \leq k \leq n_j-1, \]
for $1 \leq j \leq r$. The orthogonality conditions for type I and type II multiple orthogonal polynomials give a
linear system of $|\vec{n}|$ equations for the $|\vec{n}|$ unknown coefficients of the polynomials. If the solution exists and if it is unique, then we call the multi-index $\vec{n}$ a normal index, and if all multi-indices are normal, then the measures $(\mu_1,\ldots,\mu_r)$ are a perfect system. See \cite{Aptekarev} \cite[Ch.~23]{Ismail}  
\cite[Ch.~4.3]{NikiSor} for more information on multiple orthogonal polynomials (polyorthogonal polynomials).

An important perfect system of measures was introduced by Angelesco\footnote{This is in fact Aurel Angelescu, a Romanian mathematician who wrote a PhD thesis in 1916 under supervision of Paul Appell at the Sorbonne in Paris.} in 1919 \cite{Angelesco} and later independently suggested by Nikishin \cite{Nikishin}. An Angelesco system has $r$ measures $\mu_1,\ldots,\mu_r$ where $\mu_j$ has support in an interval $\Delta_j$ and the intervals $\Delta_1,\ldots,\Delta_r$ are pairwise disjoint. Actually the intervals may be touching. Kalyagin \cite{Kalyagin} gave an explicit example of an Angelesco system which is basically a generalization of Jacobi polynomials. He considered the two
intervals $[-1,0]$ and $[0,1]$ and investigated the type II multiple orthogonal polynomials $P_{n,m}$ satisfying
\[   \int_{-1}^0  P_{n,m}(x) x^k  (1-x)^\alpha(1+x)^\beta |x|^\gamma\, dx = 0, \qquad 0 \leq k \leq n-1, \]
\[   \int_{0}^1  P_{n,m}(x) x^k  (1-x)^\alpha(1+x)^\beta |x|^\gamma\, dx = 0, \qquad 0 \leq k \leq m-1, \]
and investigated the asymptotic behavior of $P_{n,m}$ and later, with Ronveaux \cite{KalRon}, found a third order differential equation, a four term recurrence relation and the asymptotic behavior of the ratio of two
neighboring polynomials. We call these multiple orthogonal polynomials Jacobi-Angelesco polynomials. Type I Jacobi-Angelesco polynomials were only recently investigated for the case $\alpha=\beta=\gamma=0$ because they turn
up in the analysis of Alpert multiwavelets \cite{GIWVA}. In this paper we will extend these type I Legendre-Angelesco and Jacobi-Angelesco polynomials to $r$ intervals. We take a special configuration for the $r$ intervals by having one common point $0$ and placing them on an $r$-star in the complex plane, with endpoints
at the $r$ roots of unity $\omega^j$, $j=0,1,\ldots,r-1$, with $\omega = e^{2\pi i/r}$, see Figure \ref{fig:rstar}.

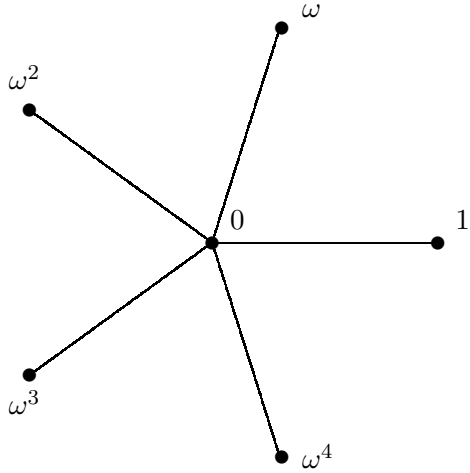
\begin{figure}[ht]
\unitlength=0.6mm
\centering
\begin{picture}(100,100)(0,0)
\put(50,50){\circle*{3}}
\put(100,50){\circle*{3}}
\put(65.45,97.55){\circle*{3}}
\put(9.55,79.39){\circle*{3}}
\put(9.55,20.61){\circle*{3}}
\put(65.45,2.45){\circle*{3}}
\put(54,53){$0$}
\put(104,53){$1$}
\put(70,100){$\omega$}
\put(5,84){$\omega^2$}
\put(5,12){$\omega^3$}
\put(70,0){$\omega^4$}
\put(50,50){\qbezier(0,0)(25,0)(50,0)}
\put(50,50){\qbezier(0,0)(7.5,23.775)(15,47.55)}
\put(50,50){\qbezier(0,0)(7.6,-23.775)(15,-47.55)}
\put(50,50){\qbezier(0,0)(-20.225,14.695)(-40.45,29.39)}
\put(50,50){\qbezier(0,0)(-20.225,-14.695)(-40.45,-29.39)}
\end{picture}
\caption{$r$-star for $r=5$ with $\omega=e^{2\pi i/5}$.}
\label{fig:rstar}
\end{figure}

To preserve the symmetry, we take a weight function $w(x) = |x|^\beta (1-x^r)^\alpha$ and the measure $\mu_j$
is supported on the interval $\Delta_j = [0,\omega^{j-1}]$, $j=1,\ldots,r$ with this weight function as its Radon-Nikodym derivative. The orthogonality properties for the type I multiple orthogonal polynomials are then
given by
\[     \sum_{j=1}^r  \int_0^{\omega^{j-1}} x^k A_{\vec{n},j}(x) |x|^\beta (1-x^r)^\alpha\, dx = 0, \qquad
    0 \leq k \leq |\vec{n}|-2, \]
and normalization
\[    \sum_{j=1}^r  \int_0^{\omega^{j-1}} x^{|\vec{n}|-1} A_{\vec{n},j}(x) |x|^\beta (1-x^r)^\alpha\, dx = 1.  \]
Observe that we are not using complex conjugation, hence the corresponding bilinear form is not an inner product.
Nevertheless the type I multiple orthogonal polynomials will exist and they are unique, at least for multi-indices on the diagonal
$\vec{n}=(n,n,\ldots,n)$ or near the diagonal $\vec{n} \pm \vec{e}_k$, where $\vec{e}_k$ is the $k$th unit vector
in $\mathbb{Z}^r$.  We will investigate these type I multiple orthogonal polynomials in Section \ref{secI}
where we give an explicit formula for the polynomials, prove their multiple orthogonality, give the recurrence
coefficients in the nearest neighbor recurrence relations and obtain a differential equation, which we use to get
the asymptotic distribution of the zeros. We give the results and the proofs for $r=2$ in full detail.
In Section \ref{secII} we consider the general case $r >1$ and again give an explicit expression for the type I multiple
orthogonal polynomials, prove their multiple orthogonality, give the recurrence coefficients of the nearest neighbor
recurrence relation near the diagonal, give a differential equation of order $r+1$, and work out the asymptotic distribution of the zeros.
The results and the proofs are more complicated and technical, and we only outline the necessary modifications of the proofs
for the case $r=2$ to general $r$.

The type II Jacobi-Angelesco polynomials are somewhat easier to analyze, because for $\vec{n} = (n,n,\ldots,n)$ they are given by a Rodrigues type
formula
\[    x^\beta (1-x^r)^\alpha P_{\vec{n}}(x) = C_n(\alpha,\beta) \frac{d^n}{dx^n} x^{\beta+n} (1-x^r)^{\alpha+n} ,  \]
where $C_n(\alpha,\beta)$ is a constant that makes $P_{\vec{n}}$ a monic polynomial. These polynomials will not be considered
in the present paper.

\section{Type I Jacobi-Angelesco polynomials for $r=2$}   \label{secI}

Type I Legendre-Angelesco polynomials appeared in \cite{GIWVA}, where they were used to expand Alpert multiwavelets.
In this case one has $r=2$  and the polynomials $(A_{n,m},B_{n,m})$ are such that
$\deg A_{n,m} = n-1$, $\deg B_{n,m} = m-1$, and the orthogonality conditions are
\[  \int_{-1}^1 \bigl( A_{n,m}(x) \chi_{[-1,0]}(x) + B_{n,m}(x) \chi_{[0,1]}(x) \bigr) x^k \, dx = 0, 
  \qquad 0 \leq k \leq n+m-2, \]
with the normalization given by
\[    \int_{-1}^1 \bigl( A_{n,m}(x) \chi_{[-1,0]}(x) + B_{n,m}(x) \chi_{[0,1]}(x) \bigr) x^{n+m-1} \, dx = 1.  \]
An explicit expression for these polynomials was given in terms of two families of polynomials $p_n$ and $q_n$
given by
\begin{equation}  \label{pn}
  p_n(x) = \sum_{k=0}^n \binom{n}{k} \binom{n+\frac{k}{2}}{n} (-1)^{n-k} x^k,  
\end{equation}
and
\begin{equation}  \label{qn}
  q_n(x) = \sum_{k=0}^n \binom{n}{k} \binom{n+\frac{k-1}{2}}{n} (-1)^{n-k} x^k.  
\end{equation}
One has (see \cite[Prop. 5 in \S 4.1]{GIWVA})
\begin{theorem}
The type I Legendre-Angelesco polynomials for multi-indices on the diagonal are given by
\[   B_{n+1,n+1}(x) = \frac12 \frac{(3n+2)!}{n!(2n+1)!} p_n(x), \quad A_{n+1,n+1}(x) = -B_{n+1,n+1}(-x), \]
and for $|n-m|=1$ one has
\[  \gamma_n B_{n+1,n}(x) = \binom{n+\frac{n}{2}}{n} q_n(x) - \binom{n+\frac{n-1}{2}}{n} p_n(x), \]
\[  \gamma_n B_{n,n+1}(x) = \binom{n+\frac{n}{2}}{n} q_n(x) + \binom{n+\frac{n-1}{2}}{n} p_n(x), \]
and
\[   A_{n+1,n}(x) = B_{n,n+1}(-x), \quad A_{n,n+1}(x) = B_{n+1,n}(-x), \]
where the normalizing constant is given by $\gamma_n = 2(\frac{n}{2}+1)_n (2n)!/(3n+1)!$.
\end{theorem}

We will extend this result by taking a more general Jacobi-type weight function and by using integration on an $r$-star
with $r >2$.

\subsection{Explicit expression}
Let us first consider the case $r=2$ and the weight function $w(x) = |x|^\beta (1-x^2)^\alpha$. The type I Jacobi-Angelesco
polynomials $(A_{n,m},B_{n,m})$ then satisfy
\[  \int_{-1}^1 \bigl( A_{n,m}(x) \chi_{[-1,0]}(x) + B_{n,m}(x) \chi_{[0,1]}(x) \bigr) x^k  |x|^\beta (1-x^2)^\alpha\, dx = 0, 
  \qquad 0 \leq k \leq n+m-2, \]
with the normalization given by
\[    \int_{-1}^1 \bigl( A_{n,m}(x) \chi_{[-1,0]}(x) + B_{n,m}(x) \chi_{[0,1]}(x) \bigr) x^{n+m-1} |x|^\beta (1-x^2)^\alpha \, dx = 1.  \]
The polynomials $A_{n,m}$ and $B_{n,m}$ on the diagonal can be expressed in term of the polynomials
\begin{equation} \label{pnab}
  p_n(x;\alpha,\beta) = \sum_{k=0}^n \binom{n}{k} \frac{\Gamma(n+\alpha+\frac{\beta+k}{2}+1)}{\Gamma(n+\alpha+1)\Gamma(\frac{\beta+k}{2}+1)}
     (-1)^{n-k} x^k. 
\end{equation}

\begin{theorem}  \label{thm22}
The type I Jacobi-Angelesco polynomials on the diagonal are given by
\[   B_{n+1,n+1}(x) = \frac12 \frac{(2\alpha+\beta+2n+2)_{n+1}}{n!} p_n(x;\alpha,\beta), \quad
     A_{n+1,n+1}(x) = -B_{n+1,n+1}(-x).  \]
\end{theorem}

\begin{proof}
If we take $A_{n+1,n+1}(x)=-B_{n+1,n+1}(-x)$, then the integral for the orthogonality conditions is
\begin{multline*}
 \int_{-1}^1 \bigl( A_{n+1,n+1}(x) \chi_{[-1,0]}(x) + B_{n+1,n+1}(x) \chi_{[0,1]}(x) \bigr) x^k |x|^\beta (1-x^r)^\alpha\, dx  \\ 
   = \bigl( 1 - (-1)^k \bigr) \int_0^1 B_{n+1,n+1}(x) x^k x^\beta (1-x^2)^\alpha\, dx .  
\end{multline*}
This is $0$ whenever $k$ is an even integer, so we only need to prove 
\begin{equation}  \label{odd}
  \int_0^1 p_n(x;\alpha,\beta) x^{2j+1} x^\beta (1-x^2)^\alpha \, dx = 0, \qquad 0 \leq j \leq n-1.  
\end{equation}
Take the polynomial $(1-x^2)^\ell -1 = \sum_{k=1}^\ell \binom{\ell}{k} (-1)^k x^{2k}$, then clearly $\bigl((1-x^2)^\ell-1)\bigr)/x$ is an odd polynomial
of degree $2\ell-1$ and thus it is sufficient to prove
\[   \int_0^1 p_n(x;\alpha,\beta) \frac{(1-x^2)^\ell-1}{x} x^\beta (1-x^2)^\alpha \, dx = S_\ell-S_0 = 0, \qquad 1 \leq \ell \leq n, \]
where
\[   S_\ell = \int_0^1 p_n(x;\alpha,\beta) x^{\beta-1}(1-x^2)^{\ell+\alpha} \, dx . \]
Using the expression \eqref{pnab} we find
\begin{align*}
   S_\ell &= \sum_{k=0}^n \binom{n}{k} \frac{\Gamma(n+\alpha+\frac{\beta+k}2+1)}{\Gamma(n+\alpha+1)\Gamma(\frac{\beta+k}{2}+1)} (-1)^{n-k}
             \int_0^1 x^{k+\beta-1}(1-x^2)^{\ell+\alpha}\, dx  \\
          &= \frac{\Gamma(\ell+\alpha+1)}{\Gamma(n+\alpha+1)} \sum_{k=0}^n \binom{n}{k} \frac{(-1)^{n-k}}{\beta+k} 
   \frac{\Gamma(n+\alpha+\frac{\beta+k}2+1)}{\Gamma(\ell+\alpha+\frac{\beta+k}{2}+1)} ,
\end{align*}
where we used the beta integral
\[    \int_0^1 x^{k+\beta-1}(1-x^2)^{\ell+\alpha}\, dx = \frac12 \textup{B}\Bigl(\frac{k+\beta}{2},\ell+\alpha+1\Bigr) 
                = \frac{\Gamma(\frac{k+\beta}2)\Gamma(\ell+\alpha+1)}{2 \Gamma(\frac{k+\beta}2+\ell+\alpha+1)} .  \]
From this we see that
\[  S_\ell-S_0 = \frac{1}{(\alpha+1)_n} \sum_{k=0}^n \binom{n}{k} (-1)^{n-k} 
        \bigl(\ell+\alpha+\frac{\beta+k}2+1\bigr)_{n-\ell} \frac{(\alpha+1)_\ell-\bigl(\alpha+\frac{\beta+k}2+1\bigr)_\ell}{\beta+k}, \]
where we used the Pochhammer symbol
\[   (a)_n = \frac{\Gamma(a+n)}{\Gamma(a)}.   \]
We see that for $1 \leq \ell \leq n$ this is of the form
\[   S_\ell-S_0 = \frac{1}{(\alpha+1)_n} \sum_{k=0}^n \binom{n}{k} (-1)^{n-k} \pi_{n-1}(k) , \]
where $\pi_{n-1}(k)$ is a polynomial of degree $n-1$ in $k$. Hence by \eqref{lem23a} in Lemma \ref{lem23}, which we prove right after this, 
we see that $S_\ell-S_0=0$
for $1 \leq \ell \leq n$, proving the relations \eqref{odd}.

For the normalization we need to show that
\[  \frac{2 n!}{(2n+2\alpha+\beta+2)_{n+1}} = 2 \int_0^1 p_n(x;\alpha,\beta) x^{2n+1} x^\beta (1-x^2)^\alpha\, dx = 2 (-1)^{n+1} (S_{n+1}-S_0). \]
Recall that $S_n-S_0=0$, so that $S_{n+1}-S_0 = S_{n+1}-S_n$, and
\[  S_{n+1}-S_n = - \sum_{k=0}^n \binom{n}{k} (-1)^{n-k} \frac{1}{2n+2\alpha+\beta+k+2}, \]
and then the result follows from \eqref{lem23b} in Lemma \ref{lem23}.
\end{proof}

In the proof of the previous theorem we used the following result.

\begin{lemma}  \label{lem23}
For all integers $n \geq 1$ one has
\begin{equation}   \label{lem23a}
   \sum_{k=0}^n \binom{n}{k} (-1)^{n-k} k^m = 0, \qquad 0 \leq m \leq n-1, 
\end{equation}
and
\begin{equation}  \label{lem23b}
   \sum_{k=0}^n \binom{n}{k} (-1)^k \frac{1}{t+k} = \frac{n!}{(t)_{n+1}}, \qquad t \in \mathbb{R} \setminus \{0,-1,-2,\ldots,-n\}.  
\end{equation}
\end{lemma}

\begin{proof}
From Newton's binomial formula
\[   (x+y)^n = \sum_{k=0}^n \binom{n}{k} x^k y^{n-k}, \]
we find, after differentiating $m$ times with respect to $x$
\[    \frac{n!}{(n-m)!} (x+y)^{n-m} = \sum_{k=m}^n \binom{n}{k} \frac{k!}{(k-m)!} x^{k-m} y^{n-k} .  \]
If we take $x=1$ and $y=-1$, then for $0 \leq m \leq n-1$
\[  \sum_{k=m}^n \binom{n}{k}(-1)^{n-k} (k-m+1)_m = 0. \]
Since $(k-m+1)_m$ is a monic polynomial of degree $m$ in $k$, this is equivalent with \eqref{lem23a}.

For the second identity we use the beta integral
\[  \int_0^1 x^{t-1}(1-x)^n\, dx = \textup{B}(t,n+1) = \frac{n!}{(t)_{n+1}}, \]
and if we expand $(1-x)^n$ then we also find
\[  \int_0^1 x^{t-1}(1-x)^n\, dx = \sum_{k=0}^n \binom{n}{k} (-1)^{k} \int_0^1 x^{k+t-1} \, dx = \sum_{k=0}^n \binom{n}{k} \frac{(-1)^k}{t+k}. \]
Comparison of both integrals gives \eqref{lem23b}.  
\end{proof}

For the type I Jacobi-Angelesco polynomials above and below the diagonal we need a second family of polynomials
\begin{equation} \label{qnab}
  q_n(x;\alpha,\beta) = \sum_{k=0}^n \binom{n}{k} \frac{\Gamma(n+\alpha+\frac{\beta+k-1}{2}+1)}{\Gamma(n+\alpha+1)\Gamma(\frac{\beta+k-1}{2}+1)}
     (-1)^{n-k} x^k. 
\end{equation}
Observe that $q_n(x;\alpha,\beta) = p_n(x;\alpha,\beta-1)$. We then have

\begin{theorem}  \label{thm24}
The type I Jacobi-Angelesco polynomials near the diagonal are given by
\begin{eqnarray}
 \gamma_n(\alpha,\beta)  B_{n+1,n}(x) &=& \nu_{n,1}(\alpha,\beta) q_n(x;\alpha,\beta) - \nu_{n,2}(\alpha,\beta) p_n(x;\alpha,\beta),  \label{B2-down} \\
 \gamma_n(\alpha,\beta) B_{n,n+1}(x) &=&  \nu_{n,1}(\alpha,\beta) q_n(x;\alpha,\beta) + \nu_{n,2}(\alpha,\beta) p_n(x;\alpha,\beta),  \label{B2-up} 
\end{eqnarray}
and
\begin{equation}  \label{A2-updown}
                 A_{n+1,n}(x) = B_{n,n+1}(-x), \quad A_{n,n+1}(x) = B_{n+1,n}(-x),
\end{equation}
where $\nu_{n,1}(\alpha,\beta)$ and $\nu_{n,2}(\alpha,\beta)$ are the leading coefficients of $p_n(x;\alpha,\beta)$ and $q_n(x;\alpha,\beta)$
respectively
\[  \nu_{n,1}(\alpha,\beta) = \frac{\Gamma\bigl(n+\alpha+\frac{\beta+n}2 +1 \bigr)}{\Gamma(n+\alpha+1) \Gamma \bigl( \frac{\beta+n}2 +1 \bigr)}, 
   \quad
     \nu_{n,2}(\alpha,\beta) = \frac{\Gamma\bigl(n+\alpha+\frac{\beta+n+1}2  \bigr)}{\Gamma(n+\alpha+1) \Gamma \bigl( \frac{\beta+n+1}2  \bigr)}, 
\]
and
\[  \gamma_n(\alpha,\beta) = \frac{2 n! \nu_{n,1}(\alpha,\beta)}{(2n+2\alpha+\beta+1)_{n+1}} .   \]
\end{theorem}

\begin{proof}
The degree of the polynomial $B_{n+1,n}$ is $n-1$ since the leading coefficients of $p_n$ and $q_n$ are cancelled by subtracting the polynomials, 
but the degree of $B_{n,n+1}$ is $n$ since we add the polynomials now. So it remains to prove the orthogonality. By \eqref{A2-updown} we see that
\begin{multline*}
 \int_{-1}^0 x^k A_{n+1,n}(x) |x|^\beta (1-x^2)^\alpha \, dx + \int_0^1 x^k B_{n+1,n}(x) x^\beta (1-x^2)^\alpha \, dx  \\
    = (-1)^k \int_0^1 x^k B_{n,n+1}(x) x^\beta (1-x^2)^\alpha\, dx + \int_0^1 x^k B_{n+1,n}(x) x^\beta (1-x^2)^\alpha\, dx . 
\end{multline*}
By using \eqref{B2-down}--\eqref{B2-up} this is equal to
\begin{multline*}   \left( 1+(-1)^k \right) \frac{\nu_{n,1}}{\gamma_n} \int_0^1 x^k q_n(x;\alpha,\beta) x^\beta (1-x^2)^\alpha\, dx  \\
   - \left( 1-(-1)^k \right) \frac{\nu_{n,2}}{\gamma_n} \int_0^1 x^k p_n(x;\alpha,\beta) x^\beta (1-x^2)^\alpha\, dx .  
\end{multline*}
When $k=2j+1$ is odd $(0 \leq j \leq n-1$) this is
\[  -2 \frac{\nu_{n,2}}{\gamma_n} \int_0^1 x^{2j+1} p_n(x;\alpha,\beta) x^\beta (1-x^2)^\alpha\, dx , \]
which vanishes because of \eqref{odd}. When $k=2j$ is even $(0 \leq j \leq n-1)$ the integral reduces to
\[   2 \frac{\nu_{n,1}}{\gamma_n} \int_0^1 x^{2j} q_n(x;\alpha,\beta) x^\beta (1-x^2)^\alpha\, dx  , \]
and this vanishes because $q_n(x;\alpha,\beta) = p_n(x;\alpha,\beta-1)$ and again by \eqref{odd}. This proves the orthogonality.
For the normalization we need
\begin{multline*}
 1 = \int_{-1}^0 x^{2n} A_{n+1,n}(x) |x|^\beta (1-x^2)^\alpha \, dx + \int_0^1 x^{2n} B_{n+1,n}(x) x^\beta (1-x^2)^\alpha \, dx  \\
    = 2 \frac{\nu_{n,1}}{\gamma_n} \int_0^1 x^{2n} q_n(x;\alpha,\beta) x^\beta (1-x^2)^\alpha\, dx . 
\end{multline*}
The latter integral is
\begin{align*}
   \int_0^1 x^{2n} q_n(x;\alpha,\beta) x^\beta (1-x^2)^\alpha\, dx &= \int_0^1 x^{2n+1} p_n(x;\alpha,\beta-1) x^{\beta-1}(1-x^2)^\alpha\, dx \\
      &= \frac{n!}{(2n+2\alpha+\beta+1)_{n+1}}, 
\end{align*}
which gives the normalizing constant $\gamma_n(\alpha,\beta)$. 
\end{proof}

\subsection{Recurrence relation}
Multiple orthogonal polynomials satisfy a system of linear recurrence relations connecting the nearest neighbors \cite{wva}.
For the type II multiple orthogonal polynomials they are
\begin{eqnarray*}
   xP_{n,m}(x) = P_{n+1,m}(x) + c_{n,m} P_{n,m}(x) + a_{n,m}P_{n-1,m}(x) + b_{n,m} P_{n,m-1}(x), \\
   xP_{n,m}(x) = P_{n,m+1}(x) + d_{n,m} P_{n,m}(x) + a_{n,m}P_{n-1,m}(x) + b_{n,m} P_{n,m-1}(x).
\end{eqnarray*}
The recurrence relations are very similar for the type I multiple orthogonal polynomials: 
\begin{eqnarray*}
   xQ_{n,m}(x) = Q_{n-1,m}(x) + c_{n-1,m} Q_{n,m}(x) + a_{n,m}Q_{n+1,m}(x) + b_{n,m} Q_{n,m+1}(x), \\
   xQ_{n,m}(x) = Q_{n,m-1}(x) + d_{n,m-1} Q_{n,m}(x) + a_{n,m}Q_{n+1,m}(x) + b_{n,m} Q_{n,m+1}(x).
\end{eqnarray*}
The same recurrence relation holds with $Q_{n,m}$ replaced by $A_{n,m}$ or $B_{n,m}$.
The coefficients $a_{n,m},b_{n,m}$ can be computed as 
\[   a_{n,m} = \frac{\kappa_{n,m}}{\kappa_{n+1,m}}, \quad b_{n,m} = \frac{\lambda_{n,m}}{\lambda_{n,m+1}}, \]
where $\kappa_{n,m}$ and $\lambda_{n,m}$ are the leading coefficients of $A_{n,m}$ and $B_{n,m}$ respectively. This can easily be seen by
checking the leading coefficients in the recurrence relation.  If we write
\[    A_{n,m}(x) = \kappa_{n,m} x^{n-1} + \delta_{n,m} x^{n-2} + \cdots, \quad  B_{n,m}(x) = \lambda_{n,m} x^{m-1} + \epsilon_{n,m} x^{m-2} + \cdots, \]
then one also has
\[    c_{n-1,m} = \frac{\delta_{n,m}}{\kappa_{n,m}} - \frac{\delta_{n+1,m}}{\kappa_{n+1,m}} - 
     \frac{\lambda_{n,m}}{\lambda_{n,m+1}} \frac{\kappa_{n,m+1}}{\kappa_{n,m}}, \]
\[   d_{n,m-1} = \frac{\epsilon_{n,m}}{\lambda_{n,m}} - \frac{\epsilon_{n,m+1}}{\lambda_{n,m+1}} - 
     \frac{\lambda_{n+1,m}}{\lambda_{n,m}} \frac{\kappa_{n,m}}{\kappa_{n+1,m}}. \] 
For the Jacobi-Angelesco polynomials near the diagonal we then have

\begin{proposition}
The recurrence coefficients for the Jacobi-Angelesco polynomials on the diagonal are
\[    a_{n,n} = \frac{n(n+\alpha)(2n+2\alpha+\beta)}{(3n+2\alpha+\beta+1)(3n+2\alpha+\beta)(3n+2\alpha+\beta-1)}, \quad   b_{n,n} = a_{n,n}, \]
and
\[    c_{n-1,n} = \frac{(2n+2\alpha+\beta-1) \Gamma(n+\alpha+\frac{n+\beta}{2}-1) \Gamma(\frac{n+\beta+1}{2})}
                       {(3n+2\alpha+\beta-1) \Gamma(n+\alpha+\frac{n+\beta-1}{2}) \Gamma(\frac{n+\beta}{2})}
    , \quad   d_{n,n-1} = -c_{n-1,n}  .  \]
\end{proposition}  
 
\begin{proof}
The two leading coefficients can easily be obtained from Theorem \ref{thm22} and Theorem \ref{thm24} and give
\[   \lambda_{n+1,n+1} = \frac12 \frac{(2n+2\alpha+\beta+2)_{n+1}}{n!} \nu_{n,1}(\alpha,\beta), \quad \kappa_{n+1,n+1} = (-1)^{n+1} \lambda_{n+1,n+1},  \]
and
\[   \lambda_{n,n+1} = \frac{\nu_{n,2}(\alpha,\beta)(2n+2\alpha+\beta+1)_{n+1}}{n!}, \quad \kappa_{n+1,n} = (-1)^n \lambda_{n,n+1}. \] 
From this the coefficients $a_{n,n}$ and $b_{n,n}$ follow easily using the formulas above. 
For $c_{n-1,n}$ and $d_{n,n-1}$ one also needs the last but one leading coefficient and the calculus is a bit longer.
\end{proof}

\subsection{Differential equation}
\begin{theorem}  \label{thm:diff2}
For $\alpha,\beta >-1$ the polynomial $p_n(x;\alpha,\beta)$ given in \eqref{pnab} satisfies the third order differential equation
\begin{multline}  \label{diff2}
  x(1-x^2)y''' + \bigl(\beta+2 -(2\alpha+\beta+6)x^2\bigr) y'' + (n-1)(3n+4\alpha+2\beta+6) xy' \\
         = n(n-1)(2n+2\alpha+\beta+2) y. 
\end{multline} 
\end{theorem}

\begin{proof}
From \eqref{pnab} it is easy to see that
\begin{equation}   \label{pndif}
    p_n'(x;\alpha,\beta) = np_{n-1}(x;\alpha+1,\beta+1), 
\end{equation}
and then of course also
\begin{equation}   \label{pndif2}
   p_n''(x;\alpha,\beta) = n(n-1)p_{n-2}(\alpha+2,\beta+2). 
\end{equation}
Hence differentiation lowers the degree $n$ but increases the parameters $\alpha$ and $\beta$.
On the other hand, one has for $\alpha,\beta>1$
\begin{multline}  \label{wpndif}
   \bigl( x^\beta (1-x^2)^\alpha p_n(x;\alpha,\beta) \bigr)' = x^{\beta-1}(1-x^2)^{\alpha-1} \bigl( (2n+2\alpha+\beta)p_{n+2}(x;\alpha-2,\beta-2) \\
      -(3n+4\alpha+2\beta) x p_{n+1}(x;\alpha-1,\beta-1) \bigr).  
\end{multline}
Indeed, if we work out the left hand side, then
\[    \bigl( x^\beta (1-x^2)^\alpha p_n(x;\alpha,\beta) \bigr)' = x^{\beta-1}(1-x^2)^{\alpha-1}  \pi_{n+2}(x), \]
where $\pi_{n+2}$ is a polynomial of degree $n+2$ given by
\begin{equation}  \label{pi1}
   \pi_{n+2}(x) = \bigl( \beta(1-x^2)-2\alpha x^2 \bigr) p_n(x;\alpha,\beta) +x(1-x^2)p_n'(x;\alpha,\beta).  
\end{equation}
One can check, by comparing coefficients and using \eqref{pnab}, that
\begin{equation}  \label{pi2}
   \pi_{n+2}(x) = (2n+2\alpha+\beta)p_{n+2}(x;\alpha-2,\beta-2)
      -(3n+4\alpha+2\beta) x p_{n+1}(x;\alpha-1,\beta-1), 
\end{equation}
but alternatively one can also observe that for $\alpha,\beta >0$
\begin{eqnarray*}
   \int_0^1 x^{\beta-2} (1-x^2)^{\alpha-1} \pi_{n+2}(x) x^{2k+1} \,dx &=& \int_0^1 \bigl( x^\beta (1-x^2)^\alpha p_n(x;\alpha,\beta) \bigr)' x^{2k}\, dx \\
                            & = & - 2k \int_0^1  x^\beta (1-x^2)^\alpha p_n(x;\alpha,\beta) x^{2k-1}\, dx \\
                            & = & 0, \qquad 0 \leq k \leq n,
\end{eqnarray*} 
where we used integration by parts and the orthogonality to odd powers \eqref{odd}. Therefore $\pi_{n+2}$ is a polynomial of degree $n+2$ which
satisfies $n+1$ orthogonality conditions for odd powers with the weight $x^{\beta-2}(1-x^2)^{\alpha-1}$ on $[0,1]$, so it belongs to a linear
space of polynomials of dimension $2$ and can be written as a linear combination of two linearly independent polynomials from that space. 
For $\alpha,\beta>1$ the polynomials
$p_{n+2}(x;\alpha-2,\beta-2)$ and $xp_{n+1}(x;\alpha-1,\beta-1)$ are two such polynomials and by \eqref{odd} they are orthogonal to odd powers
$x^{2k+1}$ for $0 \leq k \leq n$ with weight $x^{\beta-2}(1-x^2)^{\alpha-1}$ on $[0,1]$, hence $\pi_{n+2}(x) = a_n p_{n+2}(x;\alpha-2,\beta-2)
+b_n xp_{n+1}(x;\alpha-1,\beta-1)$. The coefficients $a_n$ and $b_n$ can be found by comparing the leading coefficient and the constant coefficient.

To find the differential equation we multiply \eqref{pndif2} by $x^{\beta+2}(1-x^2)^{\alpha+2}$ and differentiate to find 
\begin{multline*}
  x^{\beta+1}(1-x^2)^{\alpha+1} \Bigl[(\beta+2)(1-x^2)-2x^2(\alpha+2)\Bigr] p_n''(x;\alpha,\beta) \\ 
   + x^{\beta+2}(1-x^2)^{\alpha+2} p_n'''(x;\alpha,\beta) \\
  = n(n-1)x^{\beta+1}(1-x^2)^{\alpha+1} \Bigl( (2n+2\alpha+\beta+2)p_n(x;\alpha,\beta) \\
   - (3n+4\alpha+2\beta+6)xp_{n-1}(x,\alpha+1,\beta+1) \Bigr),
\end{multline*}
where we used the property \eqref{wpndif} for the right hand side. Remove the common factor $x^{\beta+1}(1-x^2)^{\alpha+1}$ and
use \eqref{pndif}, then the differential equation \eqref{diff2} follows.
\end{proof}

\subsection{Asymptotic zero behavior}

\begin{theorem}   \label{thm:zero2}
The asymptotic zero distribution of the polynomials $p_n(x;\alpha,\beta)$ given in \eqref{pnab} is independent of $\alpha$ and $\beta$
and is given by a measure on $[0,1]$ with density
\begin{equation}  \label{u2}
   u_2(x) = \frac{\sqrt{3}}{2\pi} \frac{(1+\sqrt{1-x^2})^{1/3} + (1-\sqrt{1-x^2})^{1/3}}{x^{1/3} \sqrt{1-x^2}}, \qquad 0 < x < 1.  
\end{equation}
\end{theorem}

\begin{proof}
Let $x_{1,n},x_{2,n},\ldots,x_{n,n}$ be the zeros of $p_n(x;\alpha,\beta)$. Since this is an Angelesco system, it is known that these zeros 
(which are the zeros of $B_{n+1,n+1}$) are simple and on the interval $(0,1)$ (see, e.g, \cite[Prop.~3.4 in Ch.~4.3]{NikiSor}).
The normalized zero counting measure is 
\[  \mu_n = \frac{1}{n} \sum_{j=1}^n \delta_{x_{j,n}}, \]
and its Stieltjes transform is
\[    S_n(z) = \int_0^1 \frac{d\mu_n(x)}{z-x} = \frac{1}{n} \frac{p_n'(z;\alpha,\beta)}{p_n(z;\alpha,\beta)}.  \]
The sequence $(\mu_n)_n$ is a sequence of probability measures on the compact interval $[0,1]$, and hence by Helley's selection principle
\cite[\S 25]{billingsley}, it contains
a subsequence $(\mu_{n_k})_k$ that converges weakly to a probability measure $\mu$ on $[0,1]$, i.e.,
\[   \lim_{k \to \infty}  \int_0^1 f(x)\, d\mu_{n_k}(x) = \lim_{k \to \infty} \frac{1}{n_k} \sum_{j=1}^{n_k} f(x_{j,n_k}) = \int_0^1 f(x)\, d\mu(x), \]
for every continuous function $f$ on $[0,1]$. The weak limit $\mu$ can depend on the subsequence, but we will show it is independent of the
choice of converging subsequence.
Observe that $p_n'(z) = n p_n(z) S_n(z)$ so that
\[   p_n''(z) = np_n'(z) S_n(z) + np_n(z) S_n'(z) = n^2 p_n(z) \left( S_n^2(z) + \frac{1}{n} S_n'(z) \right), \]
and 
\[   p_n'''(z) = n^3p_n(z) \left( S_n^3(z) + \frac{3}{n} S_n(z)S_n'(z) + \frac{1}{n^2} S_n''(z) \right).  \] 
Insert this in the differential equation \eqref{diff2} then
\begin{multline}  \label{diffS}
   z(1-z^2)n^3 p_n(z) \left( S_n^3 + \frac{3}{n} S_n S_n' + \frac{1}{n^2} S_n'' \right)
     + [\beta+2-(2\alpha+\beta+6)z^2] n^2 p_n(z) \left( S_n^2 + \frac{1}{n} S_n' \right) \\
     + (n-1)( 3n+4\alpha+2\beta+6) znp_n(z) S_n - n(n-1)(2n+2\alpha+\beta+2)p_n(z) = 0.  
\end{multline}
The weak convergence of the sequence $(\mu_{n_k})_k$ to $\mu$ implies that $S_{n_k}$ converges uniformly on compact subsets of
$\mathbb{C} \setminus [0,1]$ to the Stieltjes transform $S$ of $\mu$,
\[   S(z) = \int_0^1 \frac{d\mu(x)}{z-x}. \]
But then also $S_{n_k}'$ and $S_{n_k}''$ converge uniformly on compact subsets of $\mathbb{C} \setminus [0,1]$ to $S'$ and $S''$, respectively.
Then taking the limit for $n=n_k \to \infty$ in \eqref{diffS}, after dividing by $n^3p_n(z)$, gives the algebraic equation
\[    z(1-z^2) S^3(z) + 3zS(z) - 2 = 0.  \]
Observe that this equation does not contain $\alpha$ and $\beta$ anymore. This algebraic equation has three solutions, and we need the solution that
gives a Stieltjes transform, in particular we need the solution which is analytic on $\mathbb{C} \setminus [0,1]$ and $\lim_{z \to \infty} zS(z) = 1$. 
Solving the cubic equation gives the following three solutions
\[  S_1(z) = \frac{z^4-z^2+\bigl( (-1+(1-z^2)^{-1/2}) (z^2-1)z^2 \bigr)^{2/3}}{(z^3-z^2)\bigl( (-1+(1-z^2)^{-1/2})(z^2-1)z^2 \bigr)^{1/3}}, \]
\[ S_2(z) = - \frac{(1-i\sqrt{3}) \bigl( (-1+(1-z^2)^{-1/2})(z^2-1)^2z^2 \bigr)^{2/3} + (1+i\sqrt{3})z^2(z^2-1)}
                {(2z^3-2z) \bigl( (-1+(1-z^2)^{-1/2})(z^2-1)^2z^2 \bigr)^{1/3}}, \]
\[ S_3(z) = - \frac{(1+i\sqrt{3}) \bigl( (-1+(1-z^2)^{-1/2})(z^2-1)^2z^2 \bigr)^{2/3} + (1-i\sqrt{3})z^2(z^2-1)}
                {(2z^3-2z) \bigl( (-1+(1-z^2)^{-1/2})(z^2-1)^2z^2 \bigr)^{1/3}}. \]
One can check that
\[   \lim_{z \to \infty} zS_1(z) = -2, \quad \lim_{z\to \infty} zS_2(z) = \lim_{z \to \infty} zS_3(z) = 1, \]
hence either $S_2$ or $S_3$ is the desired solution. From the Stieltjes-Perron inversion formula (or Sokhotsky-Plemelj formula)
\[   u(x) = - \frac{1}{\pi} \lim_{\epsilon \to 0+} \Im S(x+i\epsilon)  , \qquad 0 < x < 1, \]
we find that $S_2$ is the correct solution
and it is the Stieltjes transform of the density \eqref{u2}. Hence every convergent subsequence has the same limit, 
and from the Grommer-Hamburger theorem \cite{Geronimo} it follows that $\mu_n$ converges weakly to the measure with density $u_2$.
\end{proof}

\section{Type I Jacobi-Angelesco polynomials for general $r$}  \label{secII}
We now consider the type I Jacobi-Angelesco polynomials on the $r$-star for general $r \geq 1$. The results and the proofs are similar to the case
$r=2$ but they are more complicated and technical. This is why we decided to explain the case $r=2$ in detail and only give the results and the modifications in the proof for the general case in this section. The type I Jacobi-Angelesco polynomials for the multi-index $(n_1,n_2,\ldots,n_r)$
on the $r$-star (see Figure \ref{fig:rstar} 
for $r=5$) and parameters $\alpha,\beta >-1$ are given by the vector of polynomials $(A_{\vec{n},1}^{(\alpha,\beta)},\ldots,A_{\vec{n},r}^{(\alpha,\beta)})$ which is uniquely defined by
\begin{enumerate}
  \item degree conditions: the degree of $A_{\vec{n},j}^{(\alpha,\beta)}$ is $n_j-1$,
  \item orthogonality conditions 
  \begin{equation}  \label{orthor}
   \sum_{j=1}^r \int_0^{\omega^{j-1}} x^k A_{\vec{n},j}^{(\alpha,\beta)}(x) |x|^\beta (1-x^r)^\alpha\, dx = 0, \qquad 0 \leq k \leq |\vec{n}|-2, 
  \end{equation} 
  \item normalization condition:
   \begin{equation}  \label{normr}
   \sum_{j=1}^r \int_0^{\omega^{j-1}} x^{|\vec{n}|-1} A_{\vec{n},j}^{(\alpha,\beta)}(x) |x|^\beta (1-x^r)^\alpha\, dx = 1.  
   \end{equation}
\end{enumerate}
We have used the weight $|x|^\beta(1-x^r)^\alpha$ on the $r$-star so that we can use the rotational symmetry and $\omega=e^{2\pi i/r}$ 
is the primitive $r$th root of unity.
 
\subsection{Explicit expression}
The polynomials $A_{\vec{n},j}^{(\alpha,\beta)}$ on the diagonal can be expressed in terms of the polynomials
\begin{equation}  \label{pnabr}
    p_n(x;\alpha,\beta) = \sum_{k=0}^n \binom{n}{k} \frac{\Gamma(n+\alpha+\frac{\beta+k}{r}+1)}{\Gamma(n+\alpha+1)\Gamma(\frac{\beta+k}{r}+1)}
     (-1)^{n-k} x^k.
\end{equation}
Observe that for $r=2$ this coincides with \eqref{pnab}.

\begin{theorem}   \label{thm31}
The type I Jacobi-Angelesco polynomials on the diagonal $\vec{n}=(n+1,$ $n+1,\ldots,n+1)$ are given by
\[   A_{\vec{n},j}^{(\alpha,\beta)}(x) = \lambda_{n+1,r}^{(\alpha,\beta)} p_n(\omega^{-j+1}x;\alpha,\beta), \]
with normalizing constant
\[  \lambda_{n+1,r}^{(\alpha,\beta)} = \frac{1}{r} \frac{(rn+r\alpha+\beta+r)_{n+1}}{n!} .  \]
\end{theorem}

\begin{proof}
The degree conditions are clearly satisfied. The orthogonality conditions \eqref{orthor} become
\[   \lambda_{n+1,r}^{(\alpha,\beta)} \sum_{j=1}^r (\omega^{j-1})^{k+1} \int_0^1 x^k p_n(x;\alpha,\beta) x^\beta (1-x^r)^\alpha\, dx = 0 , \]
for $0 \leq k \leq rn+r-2$. Since $\omega$ is the primitive $r$th root of unity, we have
\[   \sum_{j=1}^r (\omega^{j-1})^{k+1} = \begin{cases} 0, & \textup{if } k+1 \not\equiv 0 \bmod r, \\
                                                       r, & \textup{if } k+1 \equiv 0 \bmod r, 
                                         \end{cases}  \]
so we only need to prove 
\begin{equation}   \label{modr}
   \int_0^1 x^{rj-1} p_n(x;\alpha,\beta) x^\beta (1-x^r)^\alpha\, dx = 0, \qquad 1 \leq j \leq n.
\end{equation}
We will consider the polynomials $\bigl((1-x^r)^\ell-1 \bigr)/x$ of order $r\ell-1$ and show that $S_\ell-S_0 = 0$ for $1 \leq \ell\leq n$, where
\[   S_\ell = \int_0^1  p_n(x;\alpha,\beta) x^{\beta-1} (1-x^r)^{\alpha+\ell}\, dx, \]
as we did in the proof of Theorem \ref{thm22}. It turns out that for $1 \leq \ell \leq n$
\[ S_\ell - S_0 = \frac{1}{r \Gamma(n+\alpha+1)} \sum_{k=0}^n \binom{n}{k} \pi_{n-1,\ell}\left(\frac{\beta+k}{r} \right), \]
where $\pi_{n-1,\ell}$ is a polynomial of degree $n-1$, and this vanishes because of \eqref{lem23a} in Lemma \ref{lem23}.
For the normalization we need to prove
\begin{equation}     \label{norm}
   S_{n+1}-S_0 = S_{n+1}-S_n = (-1)^{n+1} \frac{n!}{(rn+r\alpha+\beta+r)_{n+1}}, 
\end{equation}
and this follows from \eqref{lem23b} in Lemma \ref{lem23}.
\end{proof}

Next, we will show that the type I Jacobi-Angelesco polynomials above the diagonal can be written as a linear combination of $r$ polynomials
$p_n(x;\alpha,\beta-j)$ with $0 \leq j \leq r-1$.
\begin{theorem}   \label{thm32}
Let $\vec{n}=(n,n,\ldots,n)$ and $\vec{e}_k$ be the unit vector in $\mathbb{N}^r$ with 1 on the $k$th position. The type I Jacobi-Angelesco
polynomials $A_{\vec{n}+\vec{e}_k,j}^{(\alpha,\beta)}$ are given by
\begin{equation}   \label{Aplus}
  A_{\vec{n}+\vec{e}_k,j}^{(\alpha,\beta)}(x) = A_{j-k \bmod r}(\omega^{-j+1}x) \omega^{-k+1} ,
\end{equation}
where the polynomials $A_{\ell}$, $0 \leq \ell \leq r-1$ are given by
\begin{equation}   \label{Aell}
   \tau_{n,r}^{(\alpha,\beta)} A_\ell(x) = \sum_{j=0}^{r-1} \frac{\omega^{\ell j}}{\nu_n^{(\alpha,\beta-j)}} p_n(x;\alpha,\beta-j), 
\end{equation}
with normalizing constant
\begin{equation}   \label{tau}
    \tau_{n,r}^{(\alpha,\beta)} = \frac{r n! \Gamma(n+\alpha+1) \Gamma(\frac{\beta+n+1}{r}) \Gamma(rn+r\alpha+\beta+1)}
                                         {\Gamma(n+\alpha+ \frac{\beta+n+1}{r}) \Gamma(rn+n+r\alpha+\beta+2)}, 
\end{equation}
and $\nu_n^{(\alpha,\beta)}$ is the leading coefficient of $p_n(x;\alpha,\beta)$
\begin{equation}   \label{nu}
  \nu_n^{(\alpha,\beta)} = \frac{\Gamma(n+\alpha+\frac{\beta+n}{r}+1)}{\Gamma(n+\alpha+1) \Gamma(\frac{\beta+n}{r}+1)} . 
\end{equation}
\end{theorem}

\begin{proof}
We will first determine the degree of the polynomials $A_\ell$. For $\ell=0$ we see that $\deg A_0 = \deg p_n(x;\alpha,\beta-j) =n$,
which implies that $\deg A_{\vec{n}+\vec{e}_j,j}^{(\alpha,\beta)} = n$ for all $j=1,2,\ldots,r$.
For $\ell=1,2,\ldots,r-1$ the coefficient of $x^n$ on the right hand side of \eqref{Aell} is given by
\[   \sum_{j=0}^{r-1} \omega^{\ell j} = \frac{1-\omega^{r \ell}}{1-\omega^\ell} = 0. \]
Therefore for all $k \neq j$ we have $\deg A_{\vec{n}+\vec{e}_k,j}^{(\alpha,\beta)} < n$ and one can check that it is in fact $n-1$.
For the orthogonality and the normalization we need the following integral to vanish for all $\ell = 0,1,2,\ldots,rn-1$ and to be equal to one for
$\ell=rn$:
\[  \sum_{j=1}^r \int_0^{\omega^{j-1}} x^\ell A_{\vec{n}+\vec{e}_k,j}^{(\alpha,\beta)}(x) |x|^\beta (1-x^r)^\alpha\, dx ,\]
and this expression is equal to
\[  \frac{\omega^{-k+1}}{\tau_{n,r}^{(\alpha,\beta)}} \sum_{m=0}^{r-1} \frac{\omega^{m(-k+1)}}{\nu_{n}^{(\alpha,\beta-m)}} 
     \sum_{j=1}^r (\omega^{\ell+m+1})^{j-1} \int_0^1 p_n(x;\alpha,\beta-m) x^{\ell+\beta} (1-x^r)^{\alpha}\, dx .  \]
The second sum in this expression is
\[     \sum_{j=1}^r (\omega^{\ell+m+1})^{j-1} = \begin{cases} r, & \textrm{if } \ell+m+1 \equiv 0 \bmod r, \\
                                                              0, & \textrm{if } \ell+m+1 \not\equiv 0 \bmod r.
                                                 \end{cases}  \]
Therefore we need to show that 
\[   \int_0^1 p_n(x;\alpha,\beta-m) x^{rj-m-1+\beta} (1-x^r)^{\alpha}\, dx = 0, \qquad 1 \leq j \leq n, \]
and the latter follows from \eqref{modr}. For the normalization we need to show that
\begin{eqnarray*}
   1 &=& \sum_{j=1}^r \int_0^{\omega^{j-1}} x^{rn} A_{\vec{n}+\vec{e}_k,j}^{(\alpha,\beta)}(x) |x|^\beta (1-x^r)^\alpha \, dx \\
     &=& r \frac{\omega^{r(k-1)}}{\tau_{n,r}^{(\alpha,\beta)} \nu_{n}^{(\alpha,\beta-r+1)}} 
      \int_0^1 p_n(x;\alpha,\beta-r+1) x^{rn+\beta} (1-x^r)^\alpha\, dx, 
\end{eqnarray*}
and this follows from the explicit expressions \eqref{tau} and  \eqref{nu} for $\tau_{n,r}^{(\alpha,\beta)}$ and $\nu_n^{(\alpha,\beta)}$ and the expression \eqref{norm} for the integral.
\end{proof}

We also give an explicit expression for the type I Jacobi-Angelesco polynomials below the diagonal, i.e., for $A_{\vec{n}-\vec{e}_k,j}^{(\alpha,\beta)}$,
where $\vec{n}=(n,n,\ldots,n)$.

\begin{theorem}   \label{thm33}
For every $r > 1$, $\alpha,\beta>-1$ and $\vec{n}=(n,n,\ldots,n)$ with $n > 0$ we have
\begin{multline}  \label{Amin}
   \gamma_{n,r}^{(\alpha,\beta)} A_{\vec{n}-\vec{e}_k,j}^{(\alpha,\beta)}(x)
     = \omega^{j-1} \nu_{n-1}^{(\alpha,\beta)} p_{n-1}(\omega^{-j+1}x;\alpha,\beta-1) \\
                    - \omega^{k-1} \nu_{n-1}^{(\alpha,\beta-1)} p_{n-1}(\omega^{-j+1}x;\alpha,\beta) , 
\end{multline}
 where the normalizing constant $\gamma_{n,r}^{(\alpha,\beta)}$ is given by
\begin{equation}  \label{gamma}
    \gamma_{n,r}^{(\alpha,\beta)} = \frac{r(n-1)! \Gamma(n+\alpha+\frac{n+\beta-1}{r})}
                                              {\Gamma(n+\alpha) \Gamma(\frac{n+\beta-1}{r}+1) (rn+r\alpha+\beta-1)_n} .
\end{equation}
\end{theorem}

\begin{proof}
Observe that for $j\neq k$ the degree of $A_{\vec{n}-\vec{e}_k,j}^{(\alpha,\beta)}$ is $n-1$ but that for $j=k$ the leading term
$x^{n-1}$ vanishes and the degree is $n-2$. For the orthogonality relations we need to verify that the following integral vanishes
for $0 \leq \ell \leq rn-3$:
\[ \sum_{j=1}^r \int_0^{\omega^{j-1}} x^\ell A_{\vec{n}-\vec{e}_k,j}^{(\alpha,\beta)}(x) |x|^\beta (1-x^r)^\alpha\, dx  , \]
and this expression is equal to
\begin{multline*}
   \frac{1}{\gamma_{n,r}^{(\alpha,\beta)}} \int_0^1 x^{\ell+\beta} (1-x^r)^\alpha \Bigr( \nu_{n-1}^{(\alpha,\beta)}
   \sum_{j=1}^r (\omega^{j-1})^{\ell+2} p_{n-1}(x;\alpha,\beta-1) \\
     -\ \omega^{k-1} \nu_{n-1}^{(\alpha,\beta-1)} \sum_{j=1}^r (\omega^{j-1})^{\ell+1} p_{n-1}(x;\alpha,\beta) \Bigr)\, dx .  
\end{multline*}
The two sums involving the roots of unity $\omega$ are
\[   \sum_{j=1}^r (\omega^{\ell+2})^{j-1} = \begin{cases}   r, & \textrm{if } \ell+2 \equiv 0 \bmod r, \\
                                                            0, & \textrm{if } \ell+2 \not\equiv 0 \bmod r, 
                                            \end{cases}  \]
and
\[   \sum_{j=1}^r (\omega^{\ell+1})^{j-1} = \begin{cases}   r, & \textrm{if } \ell+1 \equiv 0 \bmod r, \\
                                                            0, & \textrm{if } \ell+1 \not\equiv 0 \bmod r, 
                                            \end{cases}  \]
so the integral vanishes whenever $\ell+2 \not\equiv 0 \bmod r$ and $\ell+1 \not\equiv 0 \bmod r$.
In case $\ell=rj-2$ for $1 \leq j \leq n-1$ we need to verify
\[    \int_0^1 x^{rj-2+\beta} (1-x^r)^\alpha p_{n-1}(x;\alpha,\beta-1)\,dx = 0, \]
and this holds because of \eqref{modr}. In case $\ell=rj-1$ for $1 \leq j \leq n-1$ we need
\[  \int_0^1 x^{rj-1+\beta} (1-x^r)^\alpha p_{n-1}(x;\alpha,\beta)\,dx = 0, \]
which again follows from \eqref{modr}. For the normalization we need to show that
\begin{eqnarray*}  
   1 &=&  \sum_{j=1}^r \int_0^{\omega^{j-1}} x^{rn-2} A_{\vec{n}-\vec{e}_k,j}^{(\alpha,\beta)}(x) |x|^\beta (1-x^r)^\alpha\, dx \\
     &=& \frac{r}{\gamma_{n,r}^{(\alpha,\beta)}} \nu_{n-1}^{(\alpha,\beta)} \int_0^1 x^{rn+\beta-2} (1-x^r)^\alpha p_{n-1}(x;\alpha,\beta-1)\, dx ,
\end{eqnarray*}
and this follows by using the normalization given in Theorem \ref{thm31}.
\end{proof}

\subsection{Recurrence relation}
For general $r$ the nearest neighbor recurrence relations for the type I multiple orthogonal polynomials are 
\[ xQ_{\vec{n}}(x) = Q_{\vec{n}-\vec{e}_k}(x) + b_{\vec{n}-\vec{e}_k,k} Q_{\vec{n}}(x) + \sum_{\ell=1}^r a_{\vec{n},\ell} Q_{\vec{n}+\vec{e}_\ell}(x), \]
for $1 \leq k \leq r$, where $Q_{\vec{n}}(x) = \sum_{j=1}^r A_{\vec{n},j}(x) \chi_{[0,\omega^{j-1}]}(x)$. 
This relation can also be stated in terms of the individual polynomials on each part of the $r$-star and one has for every $j=1,2,\ldots,r$
 \[ xA_{\vec{n},j}(x) = A_{\vec{n}-\vec{e}_k,j}(x) + b_{\vec{n}-\vec{e}_k,k} A_{\vec{n},j}(x) 
+ \sum_{\ell=1}^r a_{\vec{n},\ell} A_{\vec{n}+\vec{e}_\ell,j}(x), \]
for $1 \leq k \leq r$. For the type II multiple orthogonal polynomials the recurrence relations are
\[  xP_{\vec{n}}(x) = P_{\vec{n}+\vec{e}_k}(x) + b_{\vec{n},k} P_{\vec{n}}(x) + \sum_{\ell=1}^r a_{\vec{n},\ell} P_{\vec{n}-\vec{e}_\ell}(x), \]
for $1 \leq k \leq r$. An explicit expression for the recurrence coefficients for the Jacobi-Angelesco polynomials near the diagonal is given by

\begin{proposition}  \label{prop34}
Let $\vec{n}=(n,n,\ldots,n)$ be a diagonal multi-index for $r \geq 1$. Then
\[    a_{\vec{n},k} = a_{n,r}^{(\alpha,\beta)} \omega^{2(k-1)}, \qquad 1 \leq k \leq r, \]
where
\[  a_{n,r}^{(\alpha,\beta)} = \frac{n(n+\alpha)(rn+r\alpha+\beta)}{r(rn+n+r\alpha+\beta)(rn+n+r\alpha +\beta + 1)}
    \frac{\Gamma(\frac{\beta+n+1}{r}) \Gamma(n+\alpha+\frac{\beta+n-1}{r})}{\Gamma(\frac{\beta+n-1}{r}+1)
     \Gamma(n+\alpha+\frac{\beta+n+1}{r})} .  \]
Furthermore for $r >1$
\[   b_{\vec{n}-\vec{e}_k,k} = b_{n,r}^{(\alpha,\beta)} \omega^{k-1}, \qquad 1 \leq k \leq r, \]
where
\[  b_{n,r}^{(\alpha,\beta)} = \frac{(n+\alpha+\frac{\beta-1}{r}) \Gamma(n+\alpha+\frac{n+\beta-2}{r})
                                            \Gamma(\frac{n+\beta-1}{r}+1)}
                      {(n+\alpha+\frac{n+\beta-1}{r}) \Gamma(n+\alpha+\frac{n+\beta-1}{r})
                                            \Gamma(\frac{n+\beta-2}{r}+1)}.  \] 
\end{proposition}

\begin{proof}
If we write
\[    A_{\vec{n},k}(x) = \kappa_{\vec{n},k} x^{n-1} + \delta_{\vec{n},k} x^{n-2} + \cdots, \]
then
\[   a_{\vec{n},k} = \frac{\kappa_{\vec{n},k}}{\kappa_{\vec{n}+\vec{e}_k,k}}, \]
and this ratio can be evaluated by using  Theorem \ref{thm31} which gives
\[    \kappa_{\vec{n},k} = \frac{1}{r} \frac{(rn+r\alpha+\beta)_n}{(n-1)!} 
                          \frac{\Gamma(n+\alpha+\frac{\beta+n-1}{r})}{\Gamma(n+\alpha)\Gamma(\frac{\beta+n-1}{r}+1)}
                            \omega^{(-k+1)(n-1)}, \]
and Theorem \ref{thm32} which gives
\[   \kappa_{\vec{n}+\vec{e}_k,k} = \frac{(rn+r\alpha+\beta+1)_{n+1}}{n!} 
                      \frac{\Gamma(n+\alpha+\frac{\beta+n+1}{r})}{\Gamma(n+\alpha+1)\Gamma(\frac{\beta+n+1}{r})}
                  \omega^{(-k+1)(n+1)}.  \]
The coefficients $b_{\vec{n}-\vec{e}_k,k}$ can be computed in a similar way, but the computations are a bit longer and only the case $r>1$ is covered. For $r=1$ the computations are slightly different but in that case the result is known because this corresponds to Jacobi polynomials on $[0,1]$.
The expression for $b_{\vec{n}-\vec{e}_k,k}$ is
\[   b_{\vec{n}-\vec{e}_k,k} = \frac{\delta_{\vec{n},k}}{\kappa_{\vec{n},k}} - \frac{\delta_{\vec{n}+\vec{e}_k,k}}{\kappa_{\vec{n}+\vec{e}_k,k}}
    - \sum_{\ell=1,\ell\neq k}^r \frac{\kappa_{\vec{n},\ell}}{\kappa_{\vec{n},k}} \frac{\kappa_{\vec{n}+\vec{e}_\ell,k}}{\kappa_{\vec{n}+\vec{e}_\ell,\ell}}.   \]
From Theorem \ref{thm31} one finds
\[   \delta_{\vec{n},k} = - \frac{(rn+r\alpha+\beta)_n\Gamma(n+\alpha+\frac{n+\beta-2}{r})}
                                 {r(n-2)!\Gamma(n+\alpha)\Gamma(\frac{n+\beta-2}{r}+1)}  
                                 \omega^{(-k+1)(n-2)}, \qquad n\geq 2 ,  \]
and from Theorem \ref{thm32}
\[         \delta_{\vec{n}+\vec{e}_k,k} = - \frac{n\omega^{(-k+1)n}}{\tau_{n,r}^{(\alpha,\beta)}}
           \sum_{j=0}^{r-1} \frac{\Gamma(\frac{n+\beta-j}{r}+1)\Gamma(n+\alpha+\frac{n+\beta-1-j}{r}+1)}
                                 {\Gamma(\frac{n+\beta-1-j}{r}+1)\Gamma(n+\alpha+\frac{n+\beta-j}{r}+1)}, \]      
and for $\ell\neq k$
\[ \kappa_{\vec{n}+\vec{e}_\ell,k} = - \frac{n \omega^{-\ell+1-(k-1)(n-1)}}{\tau_{n,r}^{(\alpha,\beta)}}
         \sum_{j=0}^{r-1} \omega^{(k-\ell)j} \frac{\Gamma(\frac{n+\beta-j}{r}+1)\Gamma(n+\alpha+\frac{n+\beta-1-j}{r}+1)}
                                 {\Gamma(\frac{n+\beta-1-j}{r}+1)\Gamma(n+\alpha+\frac{n+\beta-j}{r}+1)}.  \]
Combining all these results gives the desired expression for $b_{\vec{n}-\vec{e}_k,k}$.
\end{proof}

Observe that one can easily find the asymptotic behavior of these recurrence coefficients as $n\to \infty$ by using \cite[5.11.12]{NIST}
\[   \frac{\Gamma(n+a)}{\Gamma(n+b)} \sim n^{a-b}, \qquad n \to \infty, \]
which gives
\[   \lim_{n \to \infty} a_{n,r}^{(\alpha,\beta)} = \frac{r}{(r+1)^{2+2/r}}, \quad
     \lim_{n \to \infty} b_{n,r}^{(\alpha,\beta)} = \frac{r}{(r+1)^{1+1/r}}.  \]
The limit for $b_{n,r}^{(\alpha,\beta)}$ is valid for $r>1$.   

\subsection{Differential equation}
In this section we will give a linear differential equation of order $r+1$ for the polynomial $p_n(x;\alpha,\beta)$ given in \eqref{pnabr}.
The differential equation is a combination of lowering and raising operators for these polynomials, i.e., differential operators
that lower or raise the degree of the polynomial and raise/lower the parameters $\alpha$ and $\beta$.

\begin{lemma}  \label{lem35}
For the polynomial $p_n(x;\alpha,\beta)$ given in \eqref{pnabr} one has for $\alpha,\beta>-1$ the lowering property
\begin{equation} \label{pnabrlow}
    p_n'(x;\alpha,\beta) = np_{n-1}(x;\alpha+1,\beta+1) ,
\end{equation}
and for $\alpha,\beta > r-1$ the raising property
\begin{equation}   \label{pnabrrais}
   \bigl( x^\beta(1-x^r)^\alpha p_n(x;\alpha,\beta) \bigr)' = x^{\beta-1}(1-x^r)^{\alpha-1} 
\sum_{k=1}^r a_{k,n}^{(\alpha,\beta)} x^{r-k} p_{n+k}(x;\alpha-k,\beta-k),
\end{equation}  
where
\[   a_{k,n}^{(\alpha,\beta)} = (-1)^k \left[ \binom{r}{k} (r\alpha+\beta) + \binom{r+1}{k+1} kn \right] .  \]
\end{lemma}

\begin{proof}
The lowering property \eqref{pnabrlow} can easily be proved by differentiating the expression \eqref{pnabr}.
To prove \eqref{pnabrrais} we first observe that
\begin{equation}   \label{raisPi}
  \bigl[ x^\beta(1-x^r)^\alpha p_n(x;\alpha,\beta) \bigr]' = x^{\beta-1}(1-x^r)^{\alpha-1}  \pi_{n+r}(x), 
\end{equation}
where $\pi_{n+r}$ is a polynomial of degree $n+r$ given by
\begin{equation}   \label{Pinr}
  \pi_{n+r}(x) = \bigl( \beta(1-x^r)-\alpha r x^r \bigr) p_n(x;\alpha,\beta) + x(1-x^r) p_n'(x;\alpha,\beta).
\end{equation}
Integrating by parts shows that for $\alpha,\beta >0$
\begin{eqnarray*}
   \int_0^1 x^{\beta-r}(1-x^r)^{\alpha-1} \pi_{n+r}(x) x^{r(j+1)-1}\, dx
    &=& \int_0^1  \bigl[ x^\beta(1-x^r)^\alpha p_n(x;\alpha,\beta) \bigr]' x^{rj}\, dx \\
    &=& - rj \int_0^1 x^\beta (1-x^r)^\alpha p_n(x;\alpha,\beta) x^{rj-1}\, dx ,
\end{eqnarray*}
and by \eqref{modr} this is zero for $1 \leq j \leq n$ but also for $j=0$. So $\pi_{n+r}$ is a polynomial of degree $n+r$
which is orthogonal to all $x^{r\ell-1}$ for $1\leq \ell \leq n+1$ with weight $x^{\beta-r}(1-x^r)^{\alpha-1}$ on the interval $[0,1]$.
These are $n+1$ orthogonality conditions. For $\alpha,\beta > r-1$ the $r$ polynomials $x^{r-k}p_{n+k}(x;\alpha-k,\beta-k)$, $1 \leq k \leq r$, have the same orthogonality
conditions, they are all of degree $n+r$ and they are linearly independent, hence they span the linear space of polynomials of degree $n+r$
with the $n+1$ orthogonality conditions. Therefore
\begin{equation}   \label{pinrsum}
   \pi_{n+r}(x) = \sum_{k=1}^r a_{k,n}^{(\alpha,\beta)} x^{r-k} p_{n+k}(x;\alpha-k,\beta-k), 
\end{equation}
for some $a_{k,n}^{(\alpha,\beta)}$, $k=1,2,\ldots,r$. To find these coefficients, one compares the coefficients of $x^{r-k}$ in the 
latter expansion and \eqref{Pinr}.
\end{proof}

With these operators one can find the differential equation.
\begin{theorem}  \label{thm36}
For any $n \in \mathbb{N}$, $r\geq 1$ and $\alpha,\beta >-1$ the polynomial $y=p_n(x;\alpha,\beta)$ satisfies the differential equation
\begin{equation}  \label{diffr}
  x(1-x^r)y^{(r+1)} + (r+\beta) y^{(r)} +\sum_{k=0}^{r} c_{k,n}^{(\alpha,\beta)} x^k y^{(k)} = 0, 
\end{equation}
where the coefficients $c_{k,n}^{(\alpha,\beta)}$ for $0 \leq k \leq r$ are given by
\begin{equation}   \label{ckn}
   c_{k,n}^{(\alpha,\beta)} = (-1)^{r+k+1} (n-r+1)_{r-k} \left[ \binom{r}{k}(r\alpha+\beta)+ \binom{r+1}{k} (rn+r-kn) \right]. 
\end{equation}
\end{theorem}

\begin{proof}
From the lowering operation \eqref{pnabrlow} one has
\[ p_n^{(r)}(x;\alpha,\beta) = \frac{n!}{(n-r)!}p_{n-r}(x;\alpha+r,\beta+r).	\]
Multiplying both sides by $x^{\beta+r}(1-x^r)^{\alpha+r}$ and differentiating gives
\begin{align*}
  &x^{\beta+r}(1-x^r)^{\alpha+r}  p_n^{(r+1)}(x;\alpha,\beta) \\ 
  &+\ x^{\beta+r-1}(1-x^r)^{\alpha+r-1} \bigl[\beta+r)(1-x^r)- r(\alpha+r)x^r \bigr] p_{n}^{(r)}(x;\alpha,\beta) \\
 = &\frac{n!}{(n-r)!}x^{\beta+r-1}(1-x^r)^{\alpha+r-1} \sum_{k=1}^r a_{k,n-r}^{(\alpha+r,\beta+r)}x^{r-k}p_{n-r+k}(x;\alpha+r-k,\beta+r-k),
\end{align*}
where we used the raising operation \eqref{pnabrrais} for the right hand side. 
Using the  lowering operation \eqref{pnabrlow} one has
\[  p_{n-r+k}(x;\alpha+r-k;\beta+r-k) = \frac{(n-r+k)!}{n!} p_{n}^{(r-k)}(x;\alpha,\beta),  \]
hence we have
\begin{multline*}
	x(1-x^r)p_n^{(r+1)}(x;\alpha,\beta) + [(\beta+r)(1-x^r)-r(\alpha+r)x^r]p_n^{(r)}(x;\alpha,\beta) \\
	- \sum_{k=1}^ra_{k,n-r}^{(\alpha+r,\beta+r)} \frac{(n-r+k)!}{(n-r)!} x^{r-k} p_n^{(r-k)}(x;\alpha,\beta)=0,
\end{multline*}
or
\begin{align*}
x(1-x^r)p_n^{(r+1)}(x;\alpha,\beta) + (\beta+r)p_n^{(r)}(x;\alpha,\beta) + \sum_{k=0}^rc_{k,n}^{(\alpha,\beta)}  x^{k} p_n^{(k)}(x;\alpha,\beta)=0,
\end{align*}
where
\[   c_{k,n}^{(\alpha,\beta)} = - a_{r-k,n-r}^{(\alpha+r,\beta+r)} \frac{(n-k)!}{(n-r)!} ,  \]
which gives \eqref{diffr} with \eqref{ckn}.
\end{proof}

\subsection{Asymptotic zero behavior}

We now investigate the asymptotic distribution of the zeros of the type I Jacobi-Angelesco polynomials $A_{\vec{n},j}^{(\alpha,\beta)}$,
 $1 \leq j \leq r$, for the multi-index $\vec{n}=(n,n,\ldots,n)$ and $n\to \infty$. From Theorem \ref{thm31} it is clear that the zeros
of $A_{\vec{n},j}^{(\alpha,\beta)}$ are copies of the zeros of $p_n(x;\alpha,\beta)$ (which are on $[0,1]$, see Lemma \ref{lem37}), 
but rotated to the interval $[0,\omega^{j-1}]$. Hence we only need to investigate the asymptotic distribution of the zeros of $p_n(x;\alpha,\beta)$ given in \eqref{pnabr}. First we prove that the zeros of $p_n(x;\alpha,\beta)$ are all on $(0,1)$ whenever $\alpha,\beta >-1$. For this we modify the standard proof of the location of zeros of orthogonal polynomials (see, e.g., \cite[Thm. 2.2.5]{Ismail}).

\begin{lemma}  \label{lem37}
Let $\alpha,\beta >-1$, then all the zeros of $p_n(x;\alpha,\beta)$ given in \eqref{pnabr} are simple and lie in the open interval $(0,1)$.
\end{lemma}

\begin{proof}
Suppose $x_1,\ldots,x_m$ are the zeros of odd multiplicity of $p_n(x;\alpha,\beta)$ that lie in $(0,1)$ and that $m < n$. Then consider
the polynomial $q_m(x) = (x^r-x_1^r)(x^r-x_2^r)\cdots(x^r-x_m^r)$. This is a polynomial of degree $rm$ with $m$ real zeros on $(0,1)$ at the points
$x_1,\ldots,x_m$ and no other sign changes on $(0,1)$. Hence $p_n(x;\alpha,\beta)q_m(x)$ has constant sign on $(0,1)$ so that
\[   \int_0^1 p_n(x;\alpha,\beta) q_m(x) x^{\beta+r-1}(1-x^r)^\alpha\, dx \neq 0. \]
But $q_m(x)x^{r-1}$ contains only powers $x^{rj-1}$ with $1 \leq j \leq m+1 \leq n$, hence by \eqref{modr} this integral is zero. This contradiction
implies that $m \geq n$ and since $p_n(x;\alpha,\beta)$ can have at most $n$ zeros on $(0,1)$, we see that $m=n$.
\end{proof}

Denote the zeros of $p_n(x;\alpha,\beta)$  by $0 < x_{1,n} < x_{2,n} < \cdots < x_{n,n} < 1$. As before we will use the normalized zero counting   
measure
\[   \mu_n = \frac{1}{n} \sum_{j=1}^n \delta_{x_{j,n}}.  \]
Then $(\mu_n)_n$ is a sequence of probability measures on $[0,1]$, and by Helley's selection principle it will contain a subsequence $(\mu_{n_k})_k$
that converges weakly to a probability measure $\mu$ on $[0,1]$. If this limit is independent of the subsequence, then we call it the asymptotic zero distribution of the zeros of $p_n(x;\alpha,\beta)$. We will investigate this by means of the Stieltjes transform
\[    S_n(z) = \int_0^1 \frac{d\mu_n(x)}{z-x} = \frac{1}{n} \frac{p_n'(z;\alpha,\beta)}{p_n(z;\alpha,\beta)}, \quad
      S(z) = \int_0^1 \frac{d\mu(x)}{z-x}  , \qquad z \in \mathbb{C} \setminus [0,1], \]
and use the Grommer-Hamburger theorem which says that $\mu_n$ converges weakly to $\mu$ if and only if $S_n$ converges uniformly on compact
subsets of $\mathbb{C} \setminus [0,1]$ to $S$ and $zS(z) \to 1$ as $z \to \infty$ (see, e.g., \cite{Geronimo}).

First we show that the weak limit of $(\mu_{n_k})_k$ has a Stieltjes transform $S$ which satisfies an algebraic equation of order $r+1$.
\begin{proposition}  \label{prop32}
Suppose $\mu_{n_k}$ converges weakly to $\mu$, then the Stieltjes transform $S$ of $\mu$ satisfies
\begin{equation}   \label{Salg}
    z(1-z^r) S^{r+1} + \sum_{\ell=0}^{r-1} (-1)^{r+\ell+1} \binom{r+1}{\ell} (r-\ell)z^\ell S^\ell = 0.  
\end{equation}
\end{proposition}

\begin{proof}
We first show that one can express the derivatives $p_n^{(j)}$ of $p_n(z;\alpha,\beta)$ in terms of $S_n$ and its derivatives:
\begin{equation}  \label{jderpn}
    p_n^{(j)}(z;\alpha,\beta) = n^j p_n(z;\alpha,\beta) \bigr[ S_n^j(z) + \frac{1}{n} G_{n,j}(S_n,S_n',\ldots,S_n^{(j-1)}) \bigr], \qquad j \geq 0,
\end{equation}
where $G_{n,j}$ is a polynomial in $j$ variables with coefficients of order $\mathcal{O}(1)$ in $n$. This polynomial is given recursively by
\begin{multline*}
   G_{n,j}(x_1,\ldots,x_j) = x_1 G_{n,j-1}(x_1,x_2,\ldots,x_{j-1}) + (j-1)x_1^{j-2}x_2 \\
   +  \frac{1}{n} \sum_{k=1}^{j-1} \frac{\partial}{\partial x_k} G_{n,j-1}(x_1,\ldots,x_{j-1}) x_{k+1}, 
\end{multline*}
with $G_{n,0}=0$. This can be proved by induction on $j$. It is obvious for $j=0$ and for $j=1$ it follows from $G_{n,1}(x)=0$ and
\[    S_n(z) =  \frac{1}{n} \frac{p_n'(z;\alpha,\beta)}{p_n(z;\alpha,\beta)}.  \]
Suppose now it holds for $j$, then for $j+1$  we have
\begin{multline*}   p_n^{(j+1)}(z;\alpha,\beta) = \bigl( p_n^{(j)}(z;\alpha,\beta) \bigr)' = n^j p_n'(z;\alpha,\beta) 
              \bigr[ S_n^j(z) + \frac{1}{n} G_{n,j}(S_n,S_n',\ldots,S_n^{(j-1)}) \bigr]  \\ 
                     + n^j p_n(z;\alpha,\beta) \bigr[ j S_n' S_n^{j-1} + \frac{1}{n} \frac{d}{dz} G_{n,j}(S_n,S_n',\ldots,S_n^{(j-1)}) \bigr]. 
\end{multline*}
Here we can use $p_n'(z;\alpha,\beta) = n p_n(z;\alpha,\beta) S_n(z)$ and the chain rule
\[  \frac{d}{dz} G_{n,j}(S_n,S_n',\ldots,S_n^{(j-1)}) = \sum_{k=1}^j S_n^{(k)} \frac{\partial}{\partial x_k} G_{n,j}(S_n,S_n',\ldots,S_n^{(j-1)}), \]
which after collecting terms gives the desired formula \eqref{jderpn} for $j+1$.

Now insert the expressions \eqref{jderpn} into the differential equation \eqref{diffr} to find
\begin{eqnarray*}
    0 &=& z(1-z^r) n^{r+1} p_n(z) \bigl[ S_n^{r+1} + \frac{1}{n} G_{n,r+1}(S_n,\ldots,S_n^{(r)}) \bigr]  \\
      & &  +\  (r + \beta)n^r p_n(z) \bigl[ S_n^r + \frac{1}{n} G_{n,r} (S_n,\ldots,S_n^{(r-1)}) \bigr] \\
      & &  + \sum_{\ell=0}^r c_\ell(n) z^\ell n^\ell p_n(z) \bigl[ S_n^\ell + \frac{1}{n} G_{n,\ell}(S_n,\ldots,S_n^{(\ell-1)}) \bigr]. 
\end{eqnarray*}
Divide by $n^{r+1} p_n(z)$ and let $n=n_k \to \infty$. From the convergence of
$S_{n_k}$ to $S$ uniformly on compact subsets of $\mathbb{C} \setminus [0,1]$, it also follows that all the derivatives $S_{n_k}^{(j)}$ converge
to $S^{(j)}$ uniformly on compact subsets of $\mathbb{C} \setminus [0,1]$. Furthermore, from \eqref{ckn} 
\[   \lim_{n \to \infty} \frac{c_\ell(n)}{n^{r-\ell+1}} = (-1)^{r+\ell+1} \binom{r+1}{\ell} (r-\ell), \]
so that in the limit we find the equation \eqref{Salg}. 
Now observe that the equation does not depend on the subsequence $(n_k)_k$ anymore, so that every convergent subsequence has a limit $S$
satisfying equation \eqref{Salg}.
\end{proof}  

By using the binomial theorem, one can find
\[     \sum_{\ell=0}^{r+1} (-1)^{r+\ell+1} \binom{r+1}{\ell} (r-\ell)z^\ell S^\ell = -(zS+r) (zS-1)^r, \]
so that the algebraic equation \eqref{Salg} simplifies to
\begin{equation}   \label{Salg1}
     zS^{r+1} - (zS+r)(zS-1)^r = 0. 
\end{equation}
This equation has $r+1$ solutions but we are interested in the solution which is a Stieltjes transform of a probability measure on $[0,1]$.
By using the Stieltjes-Perron inversion formula one can then find the asymptotic zero distribution measure $\mu$. 

\begin{theorem}  \label{thm39}
The asymptotic zero distribution of the polynomial $p_n(x;\alpha,\beta)$ as $n \to \infty$ is given by a measure which is absolutely
continuous on $[0,1]$ with a density $u_r$ given by $u_r(x) = r x^{r-1} w_r(x^r)$,
where $w_r$ is given by
\[   w_r(\hat{x}) = \frac{r+1}{\pi} \frac{1}{|\hat{x}'(\theta)|}  
   =\frac{r+1}{\pi \hat{x}} \frac{\sin \theta \sin r\theta \sin (r+1)\theta}{|(r+1)\sin r\theta - re^{i\theta} \sin(r+1)\theta|^2}, \]
where we used the change of variables
\begin{equation}   \label{hatx}
  \hat{x} = x^r = \frac{1}{c_r} \frac{\bigl( \sin(r+1)\theta \bigr)^{r+1}}{\sin \theta (\sin r\theta)^r}, \qquad
    0 < \theta < \frac{\pi}{r+1}, 
\end{equation}
and $c_r=(r+1)^{r+1}/r^r$. 
\end{theorem} 

\begin{proof}
The proof is along the same lines as in \cite{Neuschel} where the asymptotic zero distribution was obtained for Jacobi-Pi\~neiro polynomials.
The algebraic equation \eqref{Salg1} can be transformed by taking
\[    W = \frac{zS}{zS-1}, \quad   zS = \frac{W}{W-1},  \]
which gives 
\begin{equation}  \label{Walg}
   W^{r+1} - (r+1) z^r W + rz^r = 0. 
\end{equation}
We look for a solution $W$ of the form $\rho e^{i\theta}$, where $\theta$ is real and $\rho >0$. Take $z=x \in [0,1]$ and insert $W=\rho e^{i\theta}$
into \eqref{Walg} to find
\[    \rho^{r+1} e^{i(r+1)\theta} - (r+1) x^r \rho e^{i\theta} + rx^r = 0.  \]
Hence the real and imaginary parts satisfy
\begin{eqnarray} 
    \rho^{r+1} \cos(r+1)\theta - (r+1) x^r \rho \cos \theta + rx^r & = & 0, \label{Wreal} \\
     \rho^{r+1} \sin(r+1)\theta - (r+1) x^r \rho \sin \theta & = & 0.    \label{Wimag}
\end{eqnarray}
From \eqref{Wimag} we find
\[   \hat{x} := x^r = \frac{\rho^r \sin(r+1)\theta}{(r+1)\sin \theta},  \]
and using this in \eqref{Wreal}, we find
\[  \rho(\hat{x}) = \frac{r}{r+1} \frac{\sin(r+1)\theta}{\sin r\theta} .  \]
Combining both gives \eqref{hatx}. Note that when $\theta \in (0,\pi/(r+1))$ one has $\rho > 0$.
So for $\hat{x} \in [0,1]$ the equation \eqref{Walg} has a solution of the form $\rho e^{i\theta}$. Observe that also $\rho e^{-i\theta}$
is a solution and in fact
\[   W_+(\hat{x}) = \lim_{\epsilon\to 0+} W(\hat{x}+i\epsilon) = \rho(\hat{x})e^{i\theta}, \quad 
      W_-(\hat{x}) = \lim_{\epsilon\to 0+} W(\hat{x}-i\epsilon) = \rho(\hat{x})e^{-i\theta}, \]
are the boundary values of the function $W$ which is analytic on $\mathbb{C} \setminus [0,1]$.
From the Stieltjes-Perron inversion formula (or the Sokhotsky-Plemelj formula) we can compute the density $u_r$ as
\[   u_r(x) = \frac{1}{2\pi i} \bigl( S_-(x)-S_+(x) \bigr) = \frac{\rho}{\pi x} \frac{\sin \theta}{|\rho e^{i\theta}-1|^2} . \]
Writing everything in terms of $\hat{x}$ gives the required result. 
\end{proof}

\begin{figure}[ht]
\centering                  
\includegraphics[width=4in]{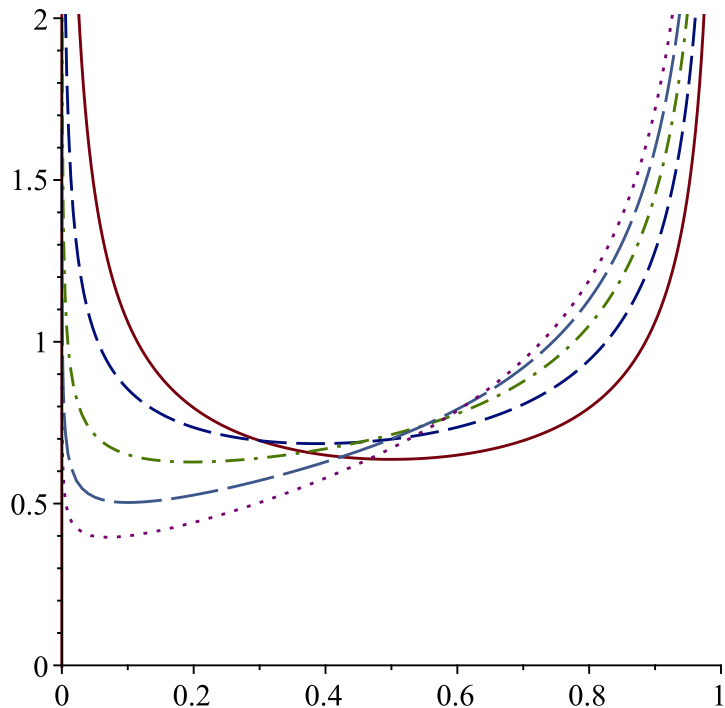}
\caption{The density $u_r$ of the asymptotic zero distribution on $[0,1]$ 
for $r=1$ (solid), $r=2$ (dash), $r=3$ (dash-dot), $r=4$ (long dash) and $r=5$ (dots).}
\label{fig:asympzero}
\end{figure}

We have plotted the densities $u_r$ on $[0,1]$ for $r=1,2,3,4,5$ in Figure \ref{fig:asympzero}. The case $r=1$ corresponds to the density
\[    u_1(x) = \frac{1}{\pi} \frac{1}{\sqrt{x(1-x)}}, \qquad 0 < x < 1 , \]
which is the well known arcsine distribution which is symmetric around $x=1/2$. For $r>1$ the symmetry is gone. The case $r=2$ corresponds to the density given in \eqref{u2}. When $x$ tends to the endpoints one has the behavior 
\begin{align*}   u_r(x) &\sim   x^{-\frac{1}{r+1}}, \qquad x \to 0, \\
                  u_r(x) &\sim  (1-x^r)^{-\frac12} , \qquad x \to 1,  
\end{align*}
so that for fixed $r$ the density $u_r$ has a singulartity of order $\frac{1}{r+1}$ at the endpoint $0$ and a singularity of order $\frac12$ at the endpoint $1$. So for $r>1$ the zeros are more dense near the endpoint $1$ than near the point $0$. This is typical for an Angelesco systems
where the zeros on all the other intervals $[0,\omega^{j-1}]$, $j=2,\ldots,r$ push the zeros on $[0,1]$ to the right.

The asymptotic behavior of the zeros is similar to the asymptotic behavior of the zeros of Jacobi-Pin\~eiro polynomials, which was studied by Neuschel and Van Assche \cite{Neuschel} and differs only by the change of variables $c_ry = x^r$.


\begin{thebibliography}{99}
\bibitem{Angelesco} A. Angelesco,
\textit{Sur deux extensions des fractions continues alg\'ebriques},
Comptes Rendus Acad. Sci. Paris \textbf{168} (1919), 262--265. 
\bibitem{Aptekarev} A.I. Aptekarev,
\textit{Multiple orthogonal polynomials},
J. Comput. Appl. Math. \textbf{99} (1998), no.~1--2, 423--447.
\bibitem{billingsley} P. Billingsley,
\textit{Probability and Measure},
John Wiley \&\ Sons, New York, 1979.
\bibitem{CCWVA} E. Coussement, J. Coussement, W. Van Assche,
\textit{Asymptotic zero distribution for a class of multiple orthogonal polynomials},
Trans. Amer. Math. Soc. \textbf{360} (2008), 5571--5588.
\bibitem{Santos} E.J.C. Dos Santos,
\textit{Monotonicity of zeros of Jacobi-Angelesco polynomials}
Proc. Amer. Math. Soc.  \textbf{145}  (2017),  no.~11, 4741--4750.
\bibitem{Geronimo} J.S. Geronimo, T.P. Hill,
\textit{Necessary and sufficient condition that the limit of Stieltjes transforms is a Stieltjes transform},
J. Approx. Theory  \textbf{121}  (2003),  no. 1, 54--60. 
\bibitem{GIWVA} J.S. Geronimo, P. Iliev, W. Van Assche,
\textit{Alpert multiwavelets and Legendre-Angelesco multiple orthogonal polynomials},
SIAM J. Math. Anal. \textbf{49} (2017), no.~1, 626--645.
\bibitem{Ismail} M.E.H. Ismail,
\textit{Classical and Quantum Orthogonal Polynomials in One Variable},
Encyclopedia of Mathematics and its Applications \textbf{98}, Cambridge University Press, 2005 (paperback edition 2009).
\bibitem{Kalyagin} V.A. Kalyagin,
\textit{A class of polynomials determined by two orthogonality relations},
Mat. Sb. (N.S.) \textbf{110} (152) (1979), no.~4, 609--627 (in Russian); 
translated in Math. USSR Sbornik \textbf{38} (1981), no.~4, 563--580.
\bibitem{KalRon} V. Kaliaguine, A. Ronveaux,
\textit{On a system of ``classical'' polynomials of simultaneous orthogonality},
J. Comput. Appl. Math. \textbf{67} (1996), no.~2, 207--217.
\bibitem{Neuschel} T. Neuschel, W. Van Assche,
\textit{Asymptotic zero distribution of Jacobi-Pi\~neiro and multiple Laguerre polynomials},
J. Approx. Theory \textbf{250} (2016), 114--132.
\bibitem{Nikishin} E.M. Nikishin,
\textit{A system of Markov functions},
Vestnik Moskov. Univ. Ser. I Mat. Mekh. \textbf{979}, no.~4, 60--63; translated in Moscow Univ. Math. Bull. \textbf{34} (1979).
\bibitem{NikiSor} E.M. Nikishin, V.N. Sorokin,
\textit{Rational Approximations and Orthogonality},
Translations of Mathematical Monographs, vol. 92, Amer. Math. Soc., Providence, RI, 1991.
\bibitem{NIST} F.W.J. Olver, D.W. Lozier, R.F. Boisvert, C.W. Clark (eds.),
\textit{NIST Handbook of Mathematical Functions}, 
NIST and Cambridge University Press, 2010.
\bibitem{wva} W. Van Assche,
\textit{Nearest neighbor recurrence relations for multiple orthogonal polynomials},
J. Approx. Theory \textbf{163} (2011), 1427--1448.
\bibitem{WVAC} W. Van Assche, E. Coussement,
\textit{Some classical multiple orthogonal polynomials},
Numerical Analysis 2000, vol. V, Quadrature and Orthogonal Polynomials,
J. Compute. Appl. Math. \textbf{127} (2001), 317--347. 
\end{thebibliography}
\end{document}